\documentclass[12pt, reqno]{amsart}

\usepackage{amsfonts,amssymb}
\usepackage{graphicx}
\usepackage{epsfig}
\usepackage{latexsym}
\usepackage{amsmath,amsthm}
\usepackage{mathrsfs}
\usepackage{graphicx} 
\usepackage{amsmath}
\usepackage{amsthm}
\usepackage{enumitem}

\usepackage{blindtext}
\usepackage{tcolorbox}
\usepackage{tikz}
\usepackage{tikz-cd}
\usepackage[all,cmtip]{xy}
\usepackage{xcolor}
\usepackage{graphicx}
\usepackage{tikz-cd}
\usepackage{comment}
\usepackage[dvipsnames]{xcolor}

\usepackage{bbm}
\usepackage[colorlinks]{hyperref}
\hypersetup{colorlinks,linkcolor={red},citecolor={blue},urlcolor={red}}
\theoremstyle{plain}
\newtheorem{theorem}{Theorem}[section]
\newtheorem{lemma}[theorem]{Lemma}

\newtheorem{remark}[theorem]{Remark}

\newtheorem{proposition}[theorem]{Proposition}
\newtheorem{corollary}[theorem]{Corollary}
\numberwithin{equation}{section}

\begin{document}
	\baselineskip=15.5pt
	
	\title{Real Slices of Parabolic $\mathrm{SL}(r,\mathbb{C})$-Opers}
		\author[S. Amrutiya]{Sanjay~Amrutiya}
	\address{Department of Mathematics, IIT Gandhinagar,
		Near Village Palaj, Gandhinagar - 382355, India}
	\email{samrutiya@iitgn.ac.in}
	\author[S. Das]{Sandipan~Das}
	\address{Department of Mathematics, IIT Gandhinagar,
		Near Village Palaj, Gandhinagar - 382355, India}
	\email{sandipan.das@iitgn.ac.in}
	
   \thanks{The research work of Sanjay Amrutiya is financially supported by the SERB-DST under project no. CRG/2023/000477. The research work of Sandipan Das is financially supported by CSIR fellowship under the scheme 09/1031(22841)/2025-EMR-I}
   
   \subjclass[2020]{14H60, 33C80, 53C07}
   
   \keywords{parabolic $\mathrm{SL}(r,\mathbb{C})$-opers, equivariant $\mathrm{SL}(r,\mathbb{C})$-opers, jet bundle, differential operator, anti-holomorphic involution, real slice}

   \begin{abstract}
   	Let $X$ be a Riemann surface equipped with an anti-holomorphic involution $\sigma_X$. We show that this induces a natural anti-holomorphic involution on the space of parabolic $\mathrm{SL}(r,\mathbb{C})$-opers. The fixed-point locus of this involution is defined as the \emph{real slice}. We further study the induced involutions on different descriptions of parabolic $\mathrm{SL}(r,\mathbb{C})$-opers, in particular differential operators, and prove that these involutions coincide.
   \end{abstract}


	\maketitle
	
	\section{Introduction}
    
    Parabolic $\mathrm{SL}(r,\mathbb{C})$-opers were introduced in \cite{biswas2020parabolic, biswas2022infinitesimal} as a natural generalization of classical opers \cite{beilinson2005opers} to the parabolic setting. As in the holomorphic case, there is a correspondence between parabolic $\mathrm{SL}(r,\mathbb{C})$-opers and differential operators with principal symbol $1$ and vanishing subprincipal symbol, as described in \cite{biswas2023parabolic}.
    
    When a Riemann surface $X$ is equipped with an anti-holomorphic involution $\sigma_X$ (real curve), this structure induces an involution on the moduli space of $\mathrm{SL}(r,\mathbb{C})$-opers. The fixed-point locus of this involution is called the \emph{real slice}, and has been studied in \cite{Biswas2022RealSO}. In this paper, we extend this framework to the parabolic setting.
    
    In Section~2, we revisit the bijective correspondence between parabolic $\mathrm{SL}(r,\mathbb{C})$-opers and differential operators between suitable parabolic vector bundles, characterized by having principal symbol $1$ and vanishing subprincipal symbol.
    
    In Section~3, we show that the existence of a real (respectively, quaternionic) theta characteristic on a real curve induces an anti-holomorphic involution on the space of parabolic $\mathrm{SL}(r,\mathbb{C})$-connections (Proposition~\ref{prop:2.2}), as well as on their automorphism group (Proposition~\ref{prop:2.4}). Furthermore, the natural action of the automorphism group on the space of parabolic $\mathrm{SL}(r,\mathbb{C})$-connections is equivariant with respect to these involutions. Consequently, this induces an involution on the space of parabolic $\mathrm{SL}(r,\mathbb{C})$-opers (Proposition~\ref{prop:2.6}).
    
    We also construct an involution on the space of differential operators
    \[
    H^0\bigl(X, \mathrm{Diff}^r_X(\mathcal{L}^{1-r}_*, \mathcal{L}^{r+1}_*)\bigr)
    \]
    induced by a real (respectively, quaternionic) theta characteristic (Corollary~\ref{corollary 2.12.1}). Our main result establishes that these two involutions coincide on the space of parabolic $\mathrm{SL}(r,\mathbb{C})$-opers.
	\section{Preliminaries}
	Throughout the paper, let $X$ be a Riemann surface endowed with an anti-holomorphic involution $\sigma_X$. Fix $S=\{x_1,x_2,\cdots,x_n\}$ be a finite set of distinct points on $X$. Sometimes, the reduced effective divisor $x_1+x_2+ \cdots+x_n$ has also been denoted by $S$.
	
	\subsection{Parabolic Vector Bundle and Parabolic Connection}
	
	A \textit{parabolic vector bundle} is a triplet $(V,\{V_{i,j}\}, \{\alpha_{i,j}\})$, where 
		\begin{itemize}
			\item $V$ is a holomorphic vector bundle of rank $n$,
			\item For each $x_i \in S$, $V_{i,j}$ is a filtration of linear subspaces of $V_{x_i}$:
			\[ V_{x_i} = V_{i,1} \supset V_{i,2} \supset ....\supset V_{i,l_i} \supset V_{i,l_{i+1}} = 0, \tag{1.1} \]
			\item $\alpha_{i,j}$ is a finite sequence of positive real numbers (parabolic weights) corresponding to the subspace $V_{i,j}$, satisfying
			\[ 0 \leq\alpha_{i,1} < \alpha_{i,2} <....< \alpha_{i,l_i}< 1 \tag{1.2}. \]
		\end{itemize}	 
	
	Now, we shall define a parabolic connection on a parabolic vector bundle.
	A \textit{logarithmic connection} on the holomorphic vector bundle $V$, singular over $S$, is a $\mathbb{C}$-linear sheaf homomorphism   \[ D: V \longrightarrow V \otimes K_X(S)  \] satisfying the Leibniz rule:
		\[D(fs) = f D(s) + s \otimes df \tag{1.3}\]

	Let $x_i \in S$ and $D$ be a logarithmic connection on $V$.
	\begin{equation} \label{1.4}
		V \xrightarrow{D} V \otimes K_X(S) \longrightarrow (V \otimes K_X(S))_{x_i} \xrightarrow{\sim} V_{x_i}. \tag{1.4}
	\end{equation}    
      $(V \otimes K_X(S))_{x_i}$ is isomorphic to $V_{x_i}$ by adjunction formula \cite{griffiths2014principles}.
	Hence, restricting the composition map at $x_i$ produces a $\mathbb{C}$-linear homomorphism of vector space:
		\[ \operatorname{Res}(D, x_i) : V_{x_i} \longrightarrow V_{x_i},\tag{1.5}\]
	which is called the \textit{residue} of the logarithmic connection $D$ at $x_i$.
	
	A \textit{parabolic connection} on $V_*$ is a logarithmic connection $D$ on $V$, singular over $S$, such that:
		\begin{itemize}
			\item $\operatorname{Res}(D, x_i)(V_{i,j}) \subseteq V_{i,j}$ for all $1 \leq j \leq l_i$, $1 \leq i \leq n$, and
			\item the endomorphism of $V_{i,j}/V_{i,j+1}$ induced by $\operatorname{Res}(D, x_i)$ coincides with multiplication by the parabolic weight $\alpha_{i,j}$ for all $1 \leq j \leq l_i$, $1 \leq i \leq n$.
		\end{itemize}
	
	\begin{proposition}\label{prop:1}
		Suppose $D$ be a parabolic connection on a parabolic vector bundle $V_*$, then $\sigma_X^*\overline{D}$ is a parabolic connection on $\sigma_X^*\overline{V_*}$.
	\end{proposition}  
	
	\begin{proof}
		Define the connection $\sigma_X^*\overline{D}$ by \[ \sigma_X^*\overline{D}(\sigma_X^*\overline{s}) := \sigma_X^*\overline{D(s)}. \] for any local section $s$ of $V$. Since $D$ is a logarithmic connection on $\sigma_X^*\overline{V_*}$ singular over $S$, it follows that $\sigma_X^*\overline{D}$ is also logarithmic with singular over $S$.
		
		\medskip
		
		Let $x_i \in S$ be a parabolic point. Suppose the parabolic filtration of $V_*$ at $\sigma_X(x_i)$ is \[ V_{\sigma_X(x_i),1} \supset V_{\sigma_X(x_i),2} \supset \cdots \supset V_{\sigma_X(x_i),l_i} \supset 0\] with parabolic weights \[\alpha_{\sigma_X(x_i),1} < \alpha_{\sigma_X(x_i),2} < \cdots <\alpha_{\sigma_X(x_i),l_i}<1. \] 
		
		The fiber of $\sigma_X^*\overline{V}$ at $x_i$ is \[ (\sigma_X^*\overline{V})_{x_i}=\overline{V}_{\sigma_X(x_i)}.\] Now the filtration of $\sigma_X^*\overline{V}$ at $x_i$ is: \[\overline{V}_{\sigma_X(x_i),1} \supset \overline{V}_{\sigma_X(x_i),2} \supset \cdots \supset \overline{V}_{\sigma_X(x_i),l_i} \supset 0\] with the same parabolic weights $\alpha_{\sigma_X(x_i),j}$. 
		
		\medskip
		
		It is easy to see that residue of $\sigma_X^*\overline{D}$ at $x_i$ satisfies \[\mathrm{Res}(\sigma_X^*\overline{D},x_i)=\overline{\mathrm{Res}(D,\sigma_X(x_i))}.\]  
		
		Since $D$ is a parabolic connection, $\mathrm{Res}(D,\sigma_X(x_i))$ from 
		${V}_{\sigma_X(x_i)}$ to ${V}_{\sigma_X(x_i)}$ preserves filtration. Hence the map \[\mathrm{Res}(\sigma_X^*\overline{D},x_i): (\sigma_X^*\overline{V})_{x_i}=\overline{V}_{\sigma_X(x_i)} \rightarrow (\sigma_X^*\overline{V})_{x_i}=\overline{V}_{\sigma_X(x_i)}\] preserves filtration. This completes the proof.
	\end{proof}
	
	\subsection{Parabolic Gunning Bundle}(\cite{biswas2022infinitesimal}, p. 36; \cite{biswas2023parabolic}, Theorem 2.2., p. 5)
	Fix a theta characteristic $\mathcal{L}$ on $X$ such that $\mathcal{L}^{\otimes 2} \simeq K_X$.
	
	$H^1(X,\mathrm{Hom}(\mathcal{L}^*,\mathcal{L}))= H^1(X,K_X) = H^0(X,\mathcal{O}_X)^* \text{(Serre Duality)} = \mathbb{C}.$ 
	
	Choose $1 \in \mathbb{C}$. Hence $1$ corresponds to some non-trivial extention $\widetilde{E}$ of $\mathcal{L}^{*}$ by $\mathcal{L}$. We have a short exact sequence:
	\begin{align} \label{1.6}
		0 \longrightarrow \mathcal{L} \longrightarrow \widetilde{E} \xrightarrow{p_0} \mathcal{L}^* \longrightarrow 0 \tag{1.6}
	\end{align}
	The sub-sheaf $\mathcal{L}^*(-S) \subset \mathcal{L}^*$. Take $E=p_0^{-1}(\mathcal{L}^*(-S)) \subset \widetilde{E}$. Hence $E$ fits the short exact sequence: 
	\begin{equation} \label{1.7}
		0 \longrightarrow \mathcal{L} \longrightarrow E \xrightarrow{p} \mathcal{L}^*(-S) \longrightarrow 0 \tag{1.7}
	\end{equation}
	where $p$ is the restriction of $p_0$ on $E$.  
	
	$E_*$ be the corresponding parabolic vector bundle with filtration :
	\begin{equation} \label{1.8}
		0 \subset \mathcal{L}^*(-S)_{x_i} \subset E_{x_i} \tag{1.8}
	\end{equation}  
	with parabolic weights \[ 1 > \frac{c_{i}+1}{2c_{i}+1} > \frac{c_{i}}{2c_{i}+1} \] where $c_{i} \in \mathbb{N}$. \cite{biswas2023parabolic}
	
	\subsection{Parabolic $\mathrm{SL}(r,\mathbb{C})$-Opers}		
	A parabolic $\mathrm{SL}(r,\mathbb{C})$ connection is a parabolic connection $\nabla_*$ on $\mathrm{sym}^{r-1}(E_*)$ such that the induced parabolic connection on det$(\mathrm{sym}^{r-1}(E_*)) = \mathcal{O}_X$ is the trivial connection.
    Two parabolic $\mathrm{SL}(r,\mathbb{C})$ connections are related iff they differ by an element of parabolic automorphism of parabolic bundle \( \mathrm{sym}^{r-1}(E_*). \) 
	
\textbf{Definition:} A \emph{parabolic $\mathrm{SL}(r,\mathbb{C})$-oper} is an equivalence class of parabolic $\mathrm{SL}(r,\mathbb{C})$ connections on the parabolic vector bundle $\mathrm{sym}^{r-1}(E_*)$.

	\subsection{Differential Operator and Symbol}
     Our goal is to describe the correspondence between parabolic $\mathrm{SL}(r,\mathbb{C})$-opers and differential operators between parabolic vector bundles with principal symbol $1$ and vanishing subprincipal symbol. For this, we begin by recalling the notions of jet bundles and differential operators.
	
	For any integer $k \geq 0$, \textit{$k$-th jet bundle} of $V$ is defined by
		\[  J^k(V):= p_{1*} \left( \frac{p_2^* V}{(p_2^* V) \otimes \mathcal{O}_{X \times X}(- (k+1)\Delta)} \right) \tag{1.9} \]
		
	where $p_i: X \times X \rightarrow X$ is the projection onto $i$-th component, $i=1,2$ and $\Delta$ is the reduced diagonal divisor defined by
	\[ \Delta:= \{(x,x) \mid x \in X\} .\]

   Let $V$ and $W$ be holomorphic vector bundles over $X$. The sheaf of \emph{holomorphic differential operators} of order $k$ from $V$ to $W$ is defined as
   \[
   \mathrm{Diff}^k(V, W) := \mathrm{Hom}(J^k(V), W) \cong W \otimes J^k(V)^*.
   \tag{1.10}
   \]
	
    Let $K_X$ denote the canonical bundle of the Riemann surface $X$. There is a natural short exact sequence of vector bundles (see \cite{biswas2003coupled}):
	\[
	0 \longrightarrow V \otimes K_X^{k} \longrightarrow J^{k}(V) \xrightarrow{q_V^k} J^{k-1}(V) \longrightarrow 0,
	\tag{1.11}
	\]
	where $q_V^k$ is the natural projection map given by restricting $k$-jets to $(k-1)$-jets at each point of $X$.
	
	The inclusion $V \otimes K_X^k \hookrightarrow J^k(V)$ induces a surjective homomorphism
	\[
	\mathrm{Diff}^k(V, W) \longrightarrow \mathrm{Hom}(V \otimes K_X^k, W)
	\cong \mathrm{Hom}(V, W) \otimes T_X^{\otimes k},
	\tag{1.12}
	\]
	which is called the \emph{symbol map}.
	
	Furthermore, every vector bundle $V$ satisfies the following short exact sequence (see
    \cite{MR2036690}): \[ 0 \longrightarrow J^k(V) \longrightarrow J^{1}(J^{k-1}(V)) \longrightarrow J^{k-2}(V) \otimes K_X \longrightarrow 0 \tag{1.13} \]

	\subsection{Orbifold Bundle}
	 Let $Y$ be a Riemann surface, and let $\Gamma$ be a finite group acting holomorphically and effectively on $Y$. Then the quotient $X := \Gamma \backslash Y$ admits a natural structure of a Riemann surface. Moreover, the projection map $p: Y \to X$ is a ramified Galois covering with Galois group $\mathrm{Aut}(Y/X) = \Gamma$.
         
      \textbf{Definition:}  
         	An \emph{orbifold bundle} (or a \emph{$\Gamma$-equivriant bundle}) is a holomorphic vector bundle $V$ on $Y$ endowed with a lift of action $\Gamma$ on $Y$ to $V$ satisfying the following:
         	\begin{itemize}
         		\item the projection $p: V \to Y$ is $\Gamma$-equivariant, and
         		\item For each $y \in Y$ and $g \in G$, the fiber map $V_y \rightarrow V_{gy}$ defined by $v \mapsto gv$ is a linear isomorphism. 
         	\end{itemize} 
        
	\begin{lemma}\label{lemma 1.15}
		Let $G$ be a group acting holomorphically and effectively on a Riemann surface $X$ endowed with an anti-holomorphic involution $\sigma_X$. Let $V$ be an equivariant bundle with an equivarint connection $D$
		, then $\sigma^*_X\overline{D}$ is equivariant connection on $\sigma^*_X\overline{V}$.
	\end{lemma}
	
	\begin{proof} 
		Let $V$ is $G$-equivariant. Hence action of $G$ on $V$ induces an action on $\overline{V}$ :
		\begin{align}\label{(1.14)}
			G \times \overline{V} \longrightarrow \overline{V} \quad (g,\overline{w}) \longmapsto \overline{gw}. \tag{1.14}
		\end{align}
		Let $\overline{D}$ be the induced connection on $\overline{V}$ defined by \( \overline{D}(\overline{s}) := \overline{D(s)}. \)
		\begin{align*}
			\overline{D}(g\overline{w})= \overline{gD(w)}
			&= g\overline{D}(\overline{w}).
		\end{align*} 
		Hence $\overline{D}$ is $G$-equivariant.
		
		Similarly, action of $G$ on $V$ induces an action on $\sigma^*_X\overline{V}$. Define the action:
		\begin{align}\label{(1.15)}
			G \times \sigma^*_X\overline{V} \longrightarrow \sigma^*_X\overline{V} \quad (g,\sigma^*_X\overline{s}) \longmapsto \sigma^*_X(g\overline{s})  \tag{1.15}.
		\end{align}
		Let $\sigma^*_X\overline{D}$ be the induced connection on $\sigma^*_X\overline{V}$ defined by \( \sigma^*_X\overline{D}(\sigma^*_X\overline{s})= \sigma^*_X\overline{D}(\overline{s}).\)
		Now we have
		\begin{align*}
			\sigma_X^*\overline{D}(g \sigma_X^*\overline{s})
			= \sigma_X^*\overline{D}(\sigma_X^*(g\overline{s})) 
			= \sigma_X^*(\overline{D}(g\overline{s})) 
			= \sigma_X^*(g\,\overline{D}(\overline{s})) 
			= g\,\sigma_X^*\overline{D}(\sigma_X^*\overline{s})
		\end{align*}
		This shows that $\sigma_X^*\overline{D}$ is $\Gamma$-equivariant.
	\end{proof}
	
	\begin{lemma}\label{lemma 1.16} Let $G$ be a group acting holomorphically and effectively on a Riemann surface $X$ endowed with an anti-holomorphic involution $\sigma_X$. \(\phi: V \longrightarrow V \) is an equivariant morphism between equivariant vector bundle $V$, then \( \sigma^*_X\overline{\phi}: \sigma^*_X\overline{V} \longrightarrow \sigma^*_X\overline{V} \) is equivariant. 
	\end{lemma}
	
	\begin{proof}
		Let $V$ is $G$-equivariant. Hence action of $G$ on $V$ induces an action on $\sigma^*_X\overline{V}$ as in (\ref{(1.14)}) and (\ref{(1.15)}) of Lemma \ref{lemma 1.15}.
		
		Let $\sigma^*_X\overline{\phi}$ be the induced morphism on $\sigma^*_X\overline{V}$ defined by \( \sigma^*_X\overline{\phi}(\sigma^*_X\overline{s}) := \sigma^*_X\overline{\phi(s)}. \)
		
		Then
		\begin{align*}
			\sigma_X^*\overline{\phi}(g \sigma_X^*\overline{s})
			= \sigma_X^*\overline{\phi}(\sigma_X^*(g\overline{s})) 
			= \sigma_X^*\bigl(\overline{\phi(g s)}\bigr) 
			= \sigma_X^*\bigl(\overline{g \phi(s)}\bigr) 
			= \sigma_X^*(g\,\overline{\phi(s)}) \\
			= g\,\sigma_X^*\overline{\phi}(\sigma_X^*\overline{s})
		\end{align*}
		where we used the $\Gamma$-equivariance of $\phi$. This proves the claim.
	\end{proof}
	
	\begin{corollary}\label{corollary 1.17} If \(\phi: V \longrightarrow W \) is an equivariant morphism, then \( \sigma^*_X\overline{\phi}: \sigma^*_X\overline{V} \longrightarrow \sigma^*_X\overline{W} \) is equivariant.         
	\end{corollary}  
	
	\subsection{Equivariant $\mathrm{SL}(r,\mathbb{C})$-opers:}The parabolic Gunning bundle $E_*$ is a parabolic vector bundle on $X$ with parabolic structure on $S$. Moreover, for each $x \in S$, the parabolic weights are integral multiple of $\frac{1}{2c_i+1}$, where $c_i \in \mathbb{N}$. Hence there is ramified Galois covering $\phi: Y \longrightarrow X$ such that:
	\begin{itemize}
		\item $\phi$ is unramified over $X \setminus S$,
		\item for each $y \in \phi^{-1}(x_i)$, the order of the ramification of $\phi$ at $y$ is $2c_i+1$, where $x_i \in S$.
	\end{itemize}
	Such a ramified covering $\phi$ exists \cite{MR933557}.
	Under this assumption there is an equivalence between the category of parabolic vector bundle on $X$ whose weights of flag over each parabolic point is $\frac{k}{2c_i+1}$, where $0 \leq k < 2c_i+1$ and the category of orbifold bundle on $Y$ has been discussed in \cite{mehta1980moduli}. Hence in the orbifold category we shall get an exact sequence analogous to $(\ref{1.7})$:
	 \begin{align*}
	   0 \longrightarrow \mathbb{L} \longrightarrow \mathcal{V} \longrightarrow \mathbb{L}^* \longrightarrow 0
	 \end{align*}
	 where the orbifold bundles $\mathcal{V}$ and $\mathbb{L}$ correspond to $E_*$ and  $\mathcal{L}$ respectively.
	 Hence we can define similarly $\mathrm{SL}(r,\mathbb{C})$-opers in the equivariant setup.
	
	A $\Gamma$-equivariant $\mathrm{SL}(r,\mathbb{C})$-connection is an equivariant connection on $\mathrm{sym}^{r-1}{\mathcal{V}}$ such that the induced connection on $\mathrm{det} ({\mathrm{sym}^{r-1}}(\mathcal{V})) = \mathcal{O}_Y$ is the trivial connection. Moreover, we can define an equivalance relation on the equivariant $\mathrm{SL}(r,\mathbb{C})$-connections iff they differ by a $\Gamma$-equiavariant automorphism of $\mathrm{sym}^{r-1}(\mathcal{V})$. Hence we have the notion of equivariant $\mathrm{SL}(r,\mathbb{C})$-oper.
	
	\textbf{Definition:}
			A \emph{$\Gamma$-equivariant $\mathrm{SL}(r,\mathbb{C})$-oper} is an equivalence class of equivariant $\mathrm{SL}(r,\mathbb{C})$-connection on $\mathrm{sym}^{r-1}(\mathcal{V})$.

	\subsection{Correspondence between parabolic $\mathrm{SL}(r,\mathbb{C})$-opers and parabolic differential operator:}Let $\mathbb{V}$ and $\mathbb{W}$ be two orbifold bundles on $Y$ corresponding to the parabolic vector bundles $V_*$ and $W_*$ on $X$, respectively. Then there is a natural identification described in \cite[Proposition 5.2.]{biswas2023parabolic}:
	\begin{align} \label{1.16}
		H^0\big(X, \mathrm{Diff}^k(V_*, W_*)\big):= H^0\big(Y, \mathrm{Diff}^k(\mathbb{V}, \mathbb{W})\big)^{\Gamma} = H^0(Y, \mathrm{Hom}\big(J^k(\mathbb{V}),\mathbb{W})\big)^\Gamma. \tag{1.16}
	\end{align}
	
	\begin{theorem}\label{thm 1.18}
		There is a canonical bijection between the space $\mathrm{SL}^\Gamma_Y(r)$ and \[
		\widetilde{\mathcal{U}}
		=\bigl\{\widetilde{\delta}\in H^0(Y,\mathrm{Diff}^r_Y(\mathbb{L}^{1-r},\mathbb{L}^{1+r})^\Gamma)
		\mid \sigma_{\mathrm{prin}}(\widetilde{\delta})=1,\ \sigma_{\mathrm{sub}}(\widetilde{\delta})=0\bigr\},
		\]
		where $\sigma_{\mathrm{prin}}$ and $\sigma_{\mathrm{sub}}$ denote the principal and subprincipal symbols respectively.
	\end{theorem}
	
	\begin{proof}
		Suppose $D$ be a $\Gamma$-equivariant connection on $\mathrm{sym}^{r-1}\mathcal{V}$. We have a short exact sequence defined in $(\ref{1.7})$. By our assumption, $\mathbb{L}$ and $\mathcal{V}$ are the orbifold bundles corresponding to $\mathcal{L}_*$ and $E_*$, respectively (parabolic structure of $\mathcal{L}$ is induced from $E_*$). Hence we shall get an exact sequence analogous to $(\ref{1.7})$:
		\begin{align*}\label{1.17}
			0 \longrightarrow \mathbb{L} \longrightarrow \mathcal{V} \longrightarrow \mathbb{L}^* \longrightarrow 0 \tag{1.17}
		\end{align*}
		
		Now $\mathrm{sym}^{r-1}\mathbb{L}^*$ and $\mathbb{L}^{1-r}$ are isomorphic $\Gamma$-equivariant bundle as $\mathbb{L}^*$ is a line bundle. The isomorphism and (\ref{1.17}) induce a $\Gamma$-equivariant morphism 
		\begin{align}\label{(1.18)}
			\psi: \mathrm{sym}^{r-1}{\mathcal{V}} \longrightarrow \mathbb{L}^{1-r}. \tag{1.18}
		\end{align}
		Using $D$ and $\psi$, we can have a $\Gamma$-equivariant morphism\[ \widetilde{\psi}_j: \mathrm{sym}^{r-1}{\mathcal{V}} \rightarrow J^j(\mathbb{L}^{1-r}) \] defined in ($3.3$) of \cite{MR2036690}. Moreover, lemma $3.2$ of \cite{MR2036690}, says that \[ \widetilde{\psi}_{r-1}: \mathrm{sym}^{r-1}{\mathcal{V}} \rightarrow J^{r-1}(\mathbb{L}^{1-r}) \] is an isomorphism and ($3.5$) says that \[\widetilde{\psi}_{r} \circ \widetilde{\psi}_{r-1}^{-1}:J^{r-1}(\mathbb{L}^{1-r})\rightarrow J^{r}(\mathbb{L}^{1-r})\] splits the exact sequence 
		\begin{align}\label{(1.19)}
			0 \rightarrow \mathbb{L}^{1-r}\otimes K_Y^{r}=\mathbb{L}^{1+r} \rightarrow J^{r}(\mathbb{L}^{1-r}) \rightarrow J^{r-1}(\mathbb{L}^{1-r}) \rightarrow 0.
			\tag{1.19}
		\end{align}
		Hence the splitting of (\ref{(1.19)}), gives a $\Gamma$-equivariant differential operator of order $r$
		\[ \widetilde{\delta}_{D}: J^{r}(\mathbb{L}^{1-r}) \rightarrow \mathbb{L}^{1+r}. \] of principal symbol $1$. 
		
		Conversely, suppose we have a $\Gamma$-equivariant map \[ \widetilde{\delta}: J^r(\mathbb{L}^{1-r}) \longrightarrow \mathbb{L}^{1+r} = \mathbb{L}^{1-r}\otimes K_Y^r \] with principal symbol=$1$. 
		
		Now consider the following commutative diagram of $\Gamma$-equivariant  homomorphism of jet bundles:
		\[
		\xymatrix
		{ ~ 
			& 0 \ar[d] 
			& 0 \ar[d] \\
			0 \ar[r] 
			& \mathbb{L}^{1-r}\otimes K_Y^r= \mathbb{L}^{1+r} \ar[r]^{i_1} \ar[d]
			& J^r(\mathbb{L}^{1-r}) \ar[r]^{p_1} \ar[d]^{\nu} 
			& J^{r-1}(\mathbb{L}^{1-r}) \ar[r] \ar[d]^{=}
			& 0\\
			0 \ar[r]
			& J^{r-1}(\mathbb{L}^{1-r}) \otimes K_Y \ar[r]^{i_2} \ar[d] 
			& J^1(J^{r-1}(\mathbb{L}^{1-r})) \ar[r]^{p_2} \ar[d] 
			& J^{r-1}(\mathbb{L}^{1-r}) \ar[r] 
			& 0\\
			0 \ar[r] 
			& J^{r-2}(\mathbb{L}^{1-r}) \otimes K_Y \ar[r]^{=} \ar[d]
			& J^{r-2}(\mathbb{L}^{1-r}) \otimes K_Y \ar[d]\\
			~ 
			& 0 
			& 0 }
		\]
		
		Since $\widetilde{\delta}$ has principal symbol =$1$, $\widetilde{\delta} \circ i_1 = \mathrm{id}_{\mathbb{L}^{1+r}}$. So, the exact sequence of the first row splits. 
		This gives a unique $\Gamma$-equivariant map $\kappa_1: J^{r-1}(\mathbb{L}^{1-r}) \longrightarrow J^r(\mathbb{L}^{1-r})$ such that 
		\begin{itemize}
			\item \( p_1 \circ \kappa_1 = \mathrm{id}_{J^{r-1}\mathbb{L}^{1-r}}, \label{8.1.4} \)
			\item \( \widetilde{\delta} \circ \kappa_1= 0. \)	  
		\end{itemize}
		
		It is easy to see from the diagram that $p_2 \circ \nu \circ \kappa_1 = \mathrm{id}_{J^{r-1}(\mathbb{L}^{1-r})}$. Hence $\nu \circ \kappa_1$ gives a splitting of the second row of exact sequence. So we get a unique map $\kappa_2: J^1(J^{r-1}(\mathbb{L}^{1-r})) \longrightarrow J^{r-1}(\mathbb{L}^{1-r}) \otimes K_Y $ satisfying
		\begin{itemize}
			\item \( i_2 \circ \kappa_2 = \mathrm{id}_{J^{r-1}(\mathbb{L}^{1-r}) \otimes K_Y}, \)
			\item \( \kappa_2 \circ \nu \circ \kappa_1 = 0. \) 
		\end{itemize}
		Hence we get a unique fist order differential operator 
		\[\kappa_2 \in H^0(Y,\mathrm{Diff}^1(J^{r-1}(\mathbb{L}^{1-r}),J^{r-1}(\mathbb{L}^{1-r}) \otimes K_Y)^{\Gamma}).\] with principal symbol = $1$. Since a $\Gamma$-equivariant first order differential operator from any vector bundle $V$ to $V \otimes K_Y$ with symbol 1 is a $\Gamma$-equivariant holomorphic connection on $V$, $\kappa_2$ is a $\Gamma$-equivariant holomorphic connection on $J^{r-1}(\mathbb{L}^{1-r})$. 
		
		These two are inverse of each other follows from \cite[Proposition 4.5.]{biswas2003coupled}. Two equivalent $\mathrm{SL}(r,\mathbb{C})$-connection correspond to a unique element in $\widetilde{\mathcal{U}}$, see \cite[p. 19]{biswas2003coupled}. 
		If $D$ is an $\mathrm{SL}(r,\mathbb{C})$-connection, then $\widetilde{\delta}_D$ has subprincipal symbol $0$. This follows from \cite[ Lamma $6.1.$, Lemma $5.3.$]{biswas2023parabolic}.
		
	\end{proof}
	\begin{corollary} \label{corollary 1.19}
		There is a canonical bijection between the space of parabolic $\mathrm{SL}^p_X(r)$ and \[
		\mathcal{U}
		=\bigl\{\delta \in H^0(X,\mathrm{Diff}^r_X({\mathcal{L}_*}^{1-r},\mathcal{L}_*^{1+r})
		\mid \sigma_{\mathrm{prin}}(\delta)=1,\ \sigma_{\mathrm{sub}}(\delta)=0\bigr\},
		\]
		where $\sigma_{\mathrm{prin}}$ and $\sigma_{\mathrm{sub}}$ denote the principal and subprincipal symbols respectively, see \cite[Theorem 6.2.]{biswas2023parabolic}.
	\end{corollary}
	\begin{proof}
		Follows from (\ref{1.16}) and Theorem \ref{thm 1.18}.
	\end{proof}
	
	\subsection{Real Parabolic Vector Bundle}  
	Suppose $X$ is endowed with an anti-holomorphic involution $\sigma_X$ satisfying the condition $\sigma_X(S)=S$.   
	  
	  A real holomorphic vector bundle $(V,\sigma_V)$ on $(X,\sigma_X,S)$ is called a
	  \emph{real parabolic vector bundle} \cite{amrutiya2014real} if the following conditions are satisfied:
	  \begin{itemize}
	  	\item For each $x_i \in S$, the fiber $V_{x_i}$ is equipped with a filtration
	  	\[
	  	V_{x_i} = V_{i,1} \supset V_{i,2} \supset \cdots \supset V_{i,l_i}
	  	\supset V_{i,l_i+1} = 0,
	  	\]
	  	such that
	  	\[
	  	(\sigma_V)_{x_i}(V_{i,j}) = V_{\sigma_X(i),j}.
	  	\tag{1.20}
	  	\]
	  	
	  	\item For each $x_i \in S$, there are associated parabolic weights
	  	\[
	  	0 \le \alpha_{i,1} < \alpha_{i,2} < \cdots < \alpha_{i,l_i} < 1,
	  	\]
	  	satisfying
	  	\[
	  	\alpha_{i,j} = \alpha_{\sigma_X(i),j}.
	  	\]
	  \end{itemize}

	\section{Real Slice of Parabolic $\mathrm{SL}(r,\mathbb{C})$-Opers}    
	From now on, let $\sigma_X(S)=S.$ Fix a theta characteristic $\mathcal{L}$ and $\hat{\sigma}^0_1$, where \( \hat{\sigma}^0_1 : \mathcal{L} \simeq \sigma^*_X\overline{\mathcal{L}} \) satisfying   $(\sigma^*_X\overline{\hat{\sigma}^0_1}) \circ \hat{\sigma}^0_1 \in \pm\text{id}_{\mathcal{L}}$. Such a theta characteristic exists \cite{MR286136}.
	
	\begin{proposition}\label{prop:2.1}
		If $\mathcal{L}$ is real and $c_{i}=c_{\sigma_{X}({i})}$, $E_*$ is a real parabolic vector bundle.
	\end{proposition}
	
	\begin{proof}
		First we shall see that If $\mathcal{L}$ is real, $\mathcal{L}^*(-S)$ is real.
		For all $U$ open in $X$, define:
		\begin{align*}
			\sigma^{\mathcal{O}_X(-S)}_U: \mathcal{O}_X(-S)(U) \rightarrow \mathcal{O}_X(-S)(\sigma_X(U)),\quad s \mapsto \overline{s \circ \sigma_X}.
		\end{align*}
		Since $S$ is $\sigma_X$ invariant, this defines an anti-holomorphic involution on $\mathcal{O}_X(-S)$. Hence $\mathcal{O}_X(-S)$ is real.
		
		\medskip
		
		Let $\sigma^{\mathcal{L}}: \mathcal{L} \rightarrow \mathcal{L}$ be a real structure on $\mathcal{L}$. For all $U$ open in $X$, define an anti-holomorphic involution:
		\begin{align*}
			\sigma^{\mathcal{L}^*}_U: \mathcal{L}^*(U) \rightarrow \mathcal{L}^*(\sigma_X(U)),\quad f \mapsto \sigma^{\mathcal{O}_X}_U \circ f \circ \sigma^\mathcal{L}_{\sigma_X(U)}.
		\end{align*}
		It is straightforward to check that $\sigma^{\mathcal{L}^*}$ is an anti-holomorphic involution. Now $\mathcal{L}^*(-S)$ is real, follows from the fact that the tensor product of two real vector bundle is real. 
		
		From (\ref{1.7}), $E$ is obtained as an extension of $\mathcal{L}^*(-S)$ over $\mathcal{L}$. Since both $\mathcal{L}^*(-S)$ and $\mathcal{L}$ are real and the extension class is preserved by the induced involution, the extension class is defined over $\mathbb{R}$. Hence $E$ admits a real structure $\sigma^0_1: E \rightarrow \sigma_X^*\overline{E}$ such that $\sigma_X^*\overline{\sigma^0_1} \circ \sigma^0_1 = id$.
		
		Moreover, $\sigma_X^*\overline{E}$ has a parabolic structure induced from parabolic structure of $E$ via $\sigma^0_1$. Hence $\sigma^0_1$ transports the filtration of $E_{x_i}$ to the filtration of $E_{\sigma_X(x_i)}$. Now choosing same parabolic weights of $x_i$ and $\sigma_X(x_i)$, we get the result.
	\end{proof}

	Hence $\sigma^1_0$ is an isomorphism of parabolic vector bundles:
	\begin{equation} \label{2.1}
		\sigma_1^0: E_* \longrightarrow \sigma_X^*\overline{E_*} \tag{2.1}
	\end{equation}
	such that $\sigma_X^* \overline{\sigma_1^0} \circ \sigma_1^0 = \text{id}.$
	
	Again, (\ref{2.1}) induces isomorphism  
	\begin{equation} \label{2.2}
		\mathrm{sym}^{r-1}(\sigma_1^0): \mathrm{sym}^{r-1}E_* \rightarrow \mathrm{sym}^{r-1}(\sigma_X^*\overline{E_*}) \cong \sigma_X^*\overline{\mathrm{sym}^{r-1}E_*} \tag{2.2}       
	\end{equation}
	such that $\sigma_X^* \overline{\mathrm{sym}^{r-1}(\sigma_1^0)} \circ \mathrm{sym}^{r-1}(\sigma_1^0) = \text{id}.$ 
	
	\begin{remark} \label{remark: 2.1.1}
	If $\mathcal{L}$ is quaternionic and $c_{i}=c_{\sigma_{X}({i})}$, it is clear from Proposition \ref{prop:2.1} that $E_*$ will not be real parabolic. Instead, the isomorphism $\sigma_1^0: E_* \longrightarrow \sigma_X^*\overline{E_*}$
	satisfies the condition \[\sigma_X^* \overline{\sigma_1^0} \circ \sigma_1^0 = -\mathrm{id}. \tag{2.3}\] 
	Moreover, $\mathrm{sym}^{r-1}(\sigma_1^0)$ satisfies the condition \[\sigma_X^* \overline{\mathrm{sym}^{r-1}(\sigma_1^0)} \circ \mathrm{sym}^{r-1}(\sigma_1^0) = -\mathrm{id}.\tag{2.4}\]
    \end{remark}
    
	\begin{remark} \label{corollary 2.2}
		If $\mathcal{L}$ is real, $E_*$ is a real parabolic vector bundle on $X$. Moreover, $\mathcal{V}$ denotes the $\Gamma$-equivarint bundle corresponding to $E_*$. Using this correspondence, we get a $\Gamma$-equivariant isomorphism
		\begin{equation} \label{2.5}
			\widetilde{\sigma}_1^0: \mathcal{V} \longrightarrow \sigma_Y^* \overline{\mathcal{V}} \tag{2.5}
		\end{equation}
		satisfying $\sigma_Y^*\widetilde{\sigma}_1^0 \circ \widetilde{\sigma}_1^0 = \mathrm{id}.$
		Similarly, \eqref{2.5} induces the $\Gamma$-equivariant isomorphism
		\begin{equation} \label{eq:2.6}
			\mathrm{sym}^{r-1}\widetilde{\sigma}_1^0: \mathrm{sym}^{r-1}\mathcal{V} \longrightarrow \sigma_Y^* \overline{\mathrm{sym}^{r-1}\mathcal{V}} \tag{2.6}
		\end{equation}
		satisfying $\sigma_Y^*\overline{\mathrm{sym}^{r-1}\widetilde{\sigma}_1^0} \circ \mathrm{sym}^{r-1}\widetilde{\sigma}_1^0 = \mathrm{id}.$
	\end{remark}

	Suppose $C_X^p(r)$ denotes the space of parabolic $\mathrm{SL}(r,\mathbb{C})$-connection on $\mathrm{sym}^{r-1}(E_*)$. First, we shall see an anti-holomorphic involution on $C_X^p(r)$ induced from $\sigma_X$ and the real parabolic structure of $E_*$.
	
Define a map
\begin{align} \label{2.7}
	\lambda : \mathcal{C}^p_X(r) &\longrightarrow \mathcal{C}^p_X(r), \tag{2.7} \\
	D &\mapsto 
	\left(\mathrm{sym}^{r-1}(\sigma_1^0)\right)^{-1} \otimes \mathrm{id}_{K_X(S)}
	\circ \sigma_X^* \overline{D}
	\circ \mathrm{sym}^{r-1}(\sigma_1^0). \notag
\end{align}
	
	Hence we have:
	\[
	\begin{aligned}
		\xymatrix@C=4em@R=4em{
			\sigma_X^*\overline{\mathrm{Sym}^{r-1}E_*} \ar[d]_{(\mathrm{Sym}^{r-1}\sigma^0_1)^{-1}} \ar[r]^{\sigma_X^*\overline{D}} & \sigma_X^*\overline{\mathrm{Sym}^{r-1}E_*} \otimes K_X(S) \ar[d]^{(\mathrm{Sym}^{r-1}\sigma^0_1)^{-1} \otimes id_{K_X(S)}}  \\
			\mathrm{Sym}^{r-1}E_* \ar[r]_{\lambda(D)} & \mathrm{Sym}^{r-1}E_*\otimes K_X(S)} 
	\end{aligned}
	\]

	\begin{proposition} \label{prop:2.2}
		$\lambda$ is a well defined anti-holomorphic involution on $\mathcal{C}_X^p(r)$. 
	\end{proposition}
	
	\begin{proof} 
		To see $\lambda(D)$ is a logarithmic connection, is straight forward from the definition as $\sigma_X^*\overline{D}$ is a logarithmic connection. 
		
		Let $x_i \in S$. $\mathrm{sym}^{r-1}\sigma^0_1$ is a parabolic isomorphism. So it preserves the filtration at each parabolic point. Moreover, $\mathrm{Res}(\lambda(D),x_i) = \mathrm{sym}^{r-1}(\sigma_1^0)^{-1}_{x_i} \circ \mathrm{Res}(\sigma_X^*\overline{D},x_i) \circ \mathrm{sym}^{r-1}(\sigma_1^0)_{x_i}.$ Hence $\lambda(D)$ is a parabolic connection.
		
		\begin{align*}
			\lambda^2(D)
			&=
			\bigl(\mathrm{sym}^{r-1}(\sigma^0_1)^{-1}\otimes \mathrm{id}\bigr)
			\circ
			\sigma_X^*\overline{\lambda(D)}
			\circ
			\mathrm{sym}^{r-1}(\sigma^0_1) \\
			&=
			\bigl(\mathrm{sym}^{r-1}(\sigma^0_1)^{-1}\otimes \mathrm{id}\bigr)
			\circ
			\sigma_X^*\overline{
				\bigl(\mathrm{sym}^{r-1}(\sigma^0_1)^{-1}\otimes \mathrm{id}\bigr)
				\circ
				\sigma_X^*\overline{D}
				\circ
				\mathrm{sym}^{r-1}(\sigma^0_1)}
			\circ
			\mathrm{sym}^{r-1}(\sigma^0_1).
		\end{align*}
		$\lambda$ is an involution follows from (\ref{2.2}) and $\sigma_X$ is an involution. Anti-holomorphic property holds from the definition, $\sigma_X$ is anti-holomorphic.
		
		It is easy to see that if $D$ is an $\mathrm{SL}(r,\mathbb{C})$-connection, $\sigma_X^*\overline{D}$ is so. $\sigma_X^*\overline{D}$ induces a map
		\[\mathrm{det}(\sigma_X^*\overline{D}): \sigma_X^*\overline{\mathcal{O}_X} \rightarrow \sigma_X^*\overline{\mathcal{O}_X}.\] Similarly, $\mathrm{det} \bigl(\mathrm{sym}^{r-1}(\sigma_1^0) \bigr)^{-1}$ induces a map \[\mathrm{det}\bigl(\mathrm{sym}^{r-1}(\sigma_1^0)^{-1} \bigr): \sigma_X^* \overline{\mathcal{O}_X} \rightarrow \mathcal{O}_X. \] From the definition of $\lambda$, it is easy to see that 
		\begin{align*}
			\mathrm{det}(\lambda(D))= \mathrm{det} \bigl( \mathrm{sym}^{r-1}(\sigma_1^0)^{-1} \bigr) \otimes \mathrm{id} \circ \mathrm{det}(\sigma_X^*\overline{D}) \circ \mathrm{det}(\bigl(\mathrm{sym}^{r-1}(\sigma_1^0) )\bigr)
		\end{align*}
		Any holomorphic isomorphism between line bundles over a compact Riemann surface is of the non-zero constant multiple of identity map. This together with $\sigma_X^*\overline{D}$ is a $\mathrm{SL}(r,\mathbb{C})$ connection, proves $\lambda(D)$ is also a $\mathrm{SL}(r,\mathbb{C})$ connnection. 
	\end{proof}
	
	\begin{remark}\label{remark 2.2.1}
	If $\mathcal{L}$ is quaternionic, $\lambda$ is still an anti-holomorphic involution.
    \end{remark}
	
	\begin{remark}\label{remark 2.2.2}
		 As in \eqref{2.7}, there is a well-defined anti-holomorphic involution on the space of $\Gamma$-equivariant $\mathrm{SL}(r,\mathbb{C})$-connection, $\mathcal{C}_Y^\Gamma(r)$ defined by:
	   \begin{align}
	   	\widetilde{\lambda} : \mathcal{C}_Y^\Gamma(r) &\longrightarrow \mathcal{C}_Y^\Gamma(r) \tag{2.8} \\
	   	D &\longmapsto 
	   	\left( \mathrm{sym}^{\,r-1}(\widetilde{\sigma}^0_1)^{-1}\otimes id_{K_Y} 
	   	\circ \sigma_Y^*\overline{D} 
	   	\circ \mathrm{sym}^{\,r-1}(\widetilde{\sigma}^0_1) 
	   	\right) \notag.
	   \end{align}
	   
	\end{remark}

	Next we shall see an anti-holomorphic involution on $\mathrm{Aut}(sym^{r-1}(E_*))$ induced from $\sigma_X$ and real parabolic structure of $E_*$
	
	Define a map
	\begin{align} \label{2.9}
		\tau : \mathrm{Aut}( \mathrm{sym}^{r-1} E_*)  \longrightarrow \mathrm{Aut}( \mathrm{sym}^{r-1} E_*) \tag{2.9} \\
		\psi \longmapsto \left(\mathrm{sym}^{r-1}\sigma^0_1\right)^{-1} \circ \sigma_X^* \overline{\psi} \circ \mathrm{sym}^{r-1}\sigma^0_1 \notag.
	\end{align}

	Hence we have the commutative diagram:
	\[
	\begin{aligned}
		\xymatrix{
			\sigma_X^*\overline{\mathrm{sym}^{r-1}E_*} \ar[d]_{\left(\mathrm{sym}^{r-1}\sigma^0_1\right)^{-1}} \ar[r]^{\sigma_X^*\overline{\psi}} & \sigma_X^*\overline{\mathrm{sym}^{r-1}E_*} \ar[d]^{\left(\mathrm{sym}^{r-1}\sigma^0_1\right)^{-1}} \\
			\mathrm{sym}^{r-1}E_* \ar[r]_{\tau(\psi)} & \mathrm{sym}^{r-1}E_*
		}
	\end{aligned}
	\]

	\begin{proposition} \label{prop:2.4}
		$\tau$ is an anti-holomorphic involution.
	\end{proposition}
	
	\begin{proof}
		Involution follows from the definition, (\ref{2.2}) and $\sigma_X$ is an involution. Anti-holomorphic also follows from the definition and $\sigma_X$ is an anti-holomorphic.
	\end{proof}
	
	 \begin{remark}
	  If $\mathcal{L}$ is quaternionic, then $\tau$ remains an anti-holomorphic involution.
	\end{remark}
	
	\begin{remark} \label{remark 2.4.2} As in \eqref{2.9}, there is a well-defined involution on the space of $\Gamma$-equivariant automorphisms of $\mathrm{sym}^{r-1}\mathcal{V}$, given by
		\begin{align} \label{2.10}
			\widetilde{\tau} : \mathrm{Aut}^\Gamma\bigl( \mathrm{sym}^{r-1} \mathcal{V} \bigr)  
			&\longrightarrow \mathrm{Aut}^\Gamma\bigl( \mathrm{sym}^{r-1} \mathcal{V} \bigr), \tag{2.10} \\
			\psi &\longmapsto 
			\left(\mathrm{sym}^{r-1}\widetilde{\sigma}^0_1\right)^{-1}
			\circ \sigma_Y^* \overline{\psi}
			\circ \mathrm{sym}^{r-1}\widetilde{\sigma}^0_1. \notag
		\end{align}
	\end{remark}
	
	\medskip
	
	Now we shall define an action of $\mathrm{Aut}(\mathrm{sym}^{r-1} E_*)$ on $C^p_X(r)$.
	
	\begin{proposition} \label{prop:2.5}
		The map $\alpha: \mathrm{Aut}(\mathrm{sym}^{r-1} E_*) \times C^p_X(r) \to C^p_X(r)$  
		\begin{equation} \label{2.11}
			(\psi, D) \mapsto \psi * D := (\psi \otimes \mathrm{id})^{-1} \circ D \circ \psi \tag{2.11}
		\end{equation}
		 is a well-defined action.
	\end{proposition}
	
	\begin{proof}
		From the definition of $\alpha$, we have the commutative diagram
		\[
		\xymatrix@C=4em@R=4em
		{
			\mathrm{sym}^{r-1}E_*
			\ar[d]^{\psi^{-1}}
			\ar[r]^D &
			\mathrm{sym}^{r-1}E_* \otimes K_X(S)
			\ar[d]_{\psi^{-1} \otimes id_{K_X(S)}} \\
			\mathrm{sym}^{r-1}E_*
			\ar[r]^{\psi*D} &
			\mathrm{sym}^{r-1}E_* \otimes K_X(S)}
		\]
		Every holomorphic isomorphism between line bundles over a compact Riemann surface is of the non-zero constant multiple of identity map. This together with $D$ is a $\mathrm{SL}(r,\mathbb{C})$ connection prove that $\psi * D$ is also a $\mathrm{SL}(r,\mathbb{C})$ connection. This shows that the action is well-defined. 
	\end{proof}

	\begin{remark} \label{remark 2.5.1}
	As in \eqref{2.11}, we can define an action of $\Gamma$-equivariant automorphism on $\Gamma$-equivariant $\mathrm{SL}(r,\mathbb{C})$-connection $C_Y^\Gamma(r)$.
	\[ \mathrm{Aut}^{\Gamma}( \mathrm{Sym}^{r-1} \mathcal{V}) \times C^\Gamma_Y(r) \longrightarrow C^\Gamma_Y(r), \qquad (\psi,D) \longmapsto \psi*D := \psi^{-1} \otimes id_{K_Y} \circ D \circ \psi.\label{(7.1.1.)} \] 
	\end{remark}

	Let $\mathrm{SL}^p_X(r)$ denotes the space of $\mathrm{SL}(r,\mathbb{C})$-opers on $X$. From the definition, it is clear that $\mathrm{SL}^p_X(r)$ is indeed all classes of $\mathcal{C}_X^p(r)$ under the action defined above, denoted as \( C_X^p(r)/\mathrm{Aut}(\mathrm{Sym}^{r-1}E_*). \)

	\begin{proposition}\label{prop:2.6}
		The ant-holomorphic involution $\lambda$ on $\mathcal{C}_X^p(r)$ induces an involution 
		\begin{equation} \label{2.12}
			\beta: \mathrm{SL}^p_X(r) \longrightarrow \mathrm{SL}^p_X(r). \tag{2.12}
		\end{equation}

	\end{proposition} 
	
	\begin{proof} 
		It is straight-forward from definition to check that $\lambda(\psi*D)=\tau(\psi)*\lambda(D).$ This shows that the action $\alpha$ is compatible with the involutions $\lambda$ and $\tau$. Hence \[\beta : \mathrm{SL}^p_X(r) \longrightarrow \mathrm{SL}^p_X(r), \qquad [D] \longmapsto [\lambda(D)] \] is well-defined. $\beta$ is an anti-holomorphic involution as $\lambda$ is.
	\end{proof} 
	
	\begin{remark} \label{remark 2.7.1}
	If $\mathcal{L}$ is quaternionic, then $\beta$ remains an anti-holomorphic involution, as $\lambda$ is so by Remark~$\ref{remark 2.2.1}$.
	\end{remark}
	
	\begin{remark} \label{remark 2.7.2}
	As in \eqref{2.12}, there is a well-defined anti-holomorphic involution on the space of $\Gamma$-equivariant $\mathrm{SL}_r$-Opers on $Y$ \( \mathrm{SL}^\Gamma_Y(r) \) induced by $\widetilde{\lambda}$.
	
	\begin{align}\label{(2.7)}
		\widetilde{\beta} : \mathrm{SL}^\Gamma_Y(r) \longrightarrow \mathrm{SL}^\Gamma_Y(r), \qquad [D] \longmapsto [\tilde{\lambda}(D)] \tag{2.13}
	\end{align} 
    \end{remark}

	A parabolic (respectively, equivariant) $\mathrm{SL}(r,\mathbb{C})$-oper is said to be a \emph{real slice} if it is fixed by the involution $\beta$ (respectively, $\widetilde{\beta}$).

	Now we shall see the notion of real slice through anti-holomorphic involution on differential operation induced by the map $\hat{\sigma}^0_1$ on $\mathcal{L}$.
	
	Suppose $\mathcal{L}_*$ denotes the parabolic vector bundle on $\mathcal{L}$ endowed with the parabolic structure induced from $E_*$. Furthermore, we assume that $\mathcal{L}$ is real and $c_{i}=c_{\sigma_{X}({i})}$. Then from Proposition \ref{prop:2.1}, $E_*$ is real parbolic. Hence $\sigma_1^0$ preserves the filtration. From this it is easy to see that $\mathcal{L}_*$ is real parabolic. Moreover, the real parabolic structure is induced from the map $\sigma_1^0$ on $E_*$. So we have the following results in equivariant category. 
	
	\begin{proposition}\label{prop:2.8}
		Let $\mathbb{L}$ be the orbifold bundle corresponding to the parabolic line bundle $\mathcal{L}_*$. Then there exists a $\Gamma$-equivariant isomorphism
		\begin{equation} \label{2.14}
			\widetilde{\theta}^0_{1} : \mathbb{L} \xrightarrow{\sim} \sigma_Y^*\overline{\mathbb{L}} \tag{2.14}
		\end{equation}
		such that
		\[
		\sigma_Y^*\overline{\widetilde{\theta}^0_{1}} \circ \widetilde{\theta}^0_{1}
		=
		\mathrm{id}_{\mathbb{L}}.
		\]
		
	\end{proposition}
	\begin{proof}
		There is an isomorphism \[ \hat{\sigma}^0_1 : \mathcal{L} \simeq \sigma^*_X\overline{\mathcal{L}} \] satisfying   $(\sigma^*_X\overline{\hat{\sigma}^0_1}) \circ \hat{\sigma}^0_1 = \text{id}_\mathcal{L}$. $\mathcal{L}_*$ has the parabolic structure induced from $E_*$. Hence \( \sigma^*_X\overline{\mathcal{L}} \) has an induced parabolic structure. Since the divisor is $\sigma_X$ invariant and parabolic weights of $x_i$ and $\sigma_X(x_i)$ are same, we get an isomorphism between the parabolic vector bundle \[\hat{\theta^0_1}: \mathcal{L_*} \simeq \sigma_X^*\overline{\mathcal{L_*}} \] such that $\sigma_X^*\hat{\theta^0_1} \circ \hat{\theta^0_1} = \text{id}.$ The claimed statement follows follows from the categorical equivalence between parabolic bundles on $X$ and $\Gamma$-equivariant bundles on $Y$. 
	\end{proof}
	
	\begin{corollary}\label{cor: 2.9}
		There is a $\Gamma$-equivariant isomorphism 
		\begin{equation} \label{2.15}
			\widetilde{\theta}^m_{n}: J^m(\mathbb{L}^{n}) \simeq \sigma_Y^*\overline{J^m(\mathbb{L}^{n})} \tag{2.15}
		\end{equation}
		such that $\sigma_Y^*\overline{\widetilde{\theta}^m_n} \circ \widetilde{\theta}^m_n = \mathrm{id}.$
	\end{corollary}
	\begin{proof}
		It is evident from functorial property of jet and Proposition \ref{prop:2.8}.
	\end{proof}
	
	\begin{corollary}\label{cor: 2.10}
		The anti-holomorphic isomorphism $\widetilde{\sigma}^0_{1}$ induces a $\Gamma$-equivariant isomorphism 
		\begin{equation} \label{2.16}
			(\widetilde{\theta}^*)^0_{1}: \mathbb{L}^* \simeq \sigma_Y^*\overline{\mathbb{L}^*}  \tag{2.16}
		\end{equation}
		such that $\sigma_Y^*\overline{(\widetilde{\theta}^*)^0_{1}} \circ (\widetilde{\theta}^*)^0_{1} = \mathrm{id}.$
	\end{corollary}
	\begin{proof}
		From Proposition \ref{prop:2.8}, it is evident that ${\widetilde{\theta}_1^0}$ is the restriction of $\widetilde{\sigma}_1^0$ on $\mathbb{L}$. 
		\[
		\xymatrix{
			0 \ar[r] 
			& \mathbb{L} \ar[r] \ar[d]^{\widetilde{\theta}^0_1}
			& \mathcal{V} \ar[r]^{\widetilde{\phi}} \ar[d]^{\widetilde{\sigma}_1^0}
			& \mathbb{L}^* \ar[r] 
			& 0 \\
			0 \ar[r] 
			& \sigma_Y^*\overline{\mathbb{L}} \ar[r] 
			& \sigma_Y^*\overline{\mathcal{V}} \ar[r]
			& \sigma_Y^*\overline{\mathbb{L}^*} \ar[r]
			& 0
		}
		\]
		Since $\widetilde{\sigma}_1^0$ preserves $\mathbb{L}$. It induces a well-defined map on the quotient $\mathbb{L}^*$. Explicitly,
		\[ (\widetilde{\theta}^*)^0_{1}: \mathbb{L}^* \longrightarrow \sigma_Y^*\overline{\mathbb{L}^*},\quad [v] \mapsto [\widetilde{\sigma}_1^0(v)].\] Since ${\widetilde{\theta}_1^0}$ and $\widetilde{\sigma}^0_1$ is $\Gamma$-equivariant isomorphism, $(\widetilde{\theta}^*)^0_{1}$ is so. 
	\end{proof}
	
	\begin{proposition} \label{prop:2.11}
		There exists a natural conjugate-linear involution
		\begin{align} \label{eq:2.9}
			\tilde{\mathcal{B}} : H^0(Y, \mathrm{Diff}^r_Y(\mathbb{L}^{1-r}, \mathbb{L}^{r+1}))^\Gamma \longrightarrow H^0(Y, \mathrm{Diff}^r_Y(\mathbb{L}^{1-r}, \mathbb{L}^{r+1}))^\Gamma \tag{2.17} \\
			\widetilde{\delta} \longmapsto \widetilde{\mathcal{B}}(\widetilde{\delta}) 
			:= (\widetilde{\theta}^0_{1+r})^{-1} \;\circ\; 
			\sigma^*_Y \overline{\widetilde{\delta}} \;\circ\;
			\widetilde{\theta}^r_{1-r}. \notag
		\end{align}
		induced by the anti-holomorphic isomorphism on $\mathcal{L}_*$. 
	\end{proposition} 
	\begin{proof}
		We have the commutative diagram
		\[
		\xymatrix{
			J^r(\mathbb{L}^{1-r}) 
			\ar[d]_{\widetilde{\theta}^r_{1-r}} 
			\ar[r]^{\;\;\widetilde{\mathcal{B}}(\widetilde{\delta})} 
			& \mathbb{L}^{1+r} 
			\ar[d]^{\widetilde{\theta}^0_{1+r}} \\
			\sigma^*_Y \overline{J^r(\mathbb{L}^{1-r})} 
			\ar[r]_{\;\;\sigma^*_Y \overline{\widetilde{\delta}}} 
			& \sigma^*_Y \overline{\mathbb{L}^{1+r}}
		}
		\]  
		Proof is similar in spirit and follows from corollary \ref{cor: 2.9}, Lemma \ref{lemma 1.15} and Lemma \ref{lemma 1.16}. 
	\end{proof}
	
	\begin{corollary} \label{corollary 2.12.1}
		Hence we have a natural conjugate-linear involution 
		\begin{align} \label{eq:2.18}
			\mathcal{B} : H^0(X, \mathrm{Diff}^r_X(\mathcal{L}^{1-r}_*, \mathcal{L}^{r+1}_*)) \longrightarrow H^0(X, \mathrm{Diff}^r_X(\mathcal{L}^{1-r}_*, \mathcal{L}^{r+1}_*)) \tag{2.18}
		\end{align}
	\end{corollary}
	\begin{proof}
		Proof follows from $\eqref{1.16}$.
	\end{proof}
	
	\begin{remark} \label{remark 2.12.1}
	If $\mathcal{L}$ is quaternionic, $\widetilde{\mathcal{B}}$ and $\mathcal{B}$ is still a conjuagte linear involution.
	\end{remark}
	
	\begin{theorem}\label{thm 2.11}
		The involution $\widetilde{\mathcal{B}}$ defined in Proposition \ref{prop:2.11} coincides with the involution $\widetilde{\beta}$ defined in (\ref{(2.7)}) on space of $\Gamma$-equiavriant $\mathrm{SL}(r,\mathbb{C})$-opers on Y. In other words, the following diagram commutes-
		\[
		\xymatrix
		{ 
			\mathrm{SL}_Y^\Gamma(r) \ar[r]^{\widetilde{\beta}} \ar[d]^{\widetilde{\psi}}
			& \mathrm{SL}_Y^\Gamma(r) \ar[d]^{\widetilde{\psi}}\\
			\widetilde{\mathcal{U}} \ar[r]^{\widetilde{\mathcal{B}}} 
			& \widetilde{\mathcal{U}}
		}
		\]
	\end{theorem}
	
	\begin{proof}
		We first prove the compatibility of real structures. 
		
		By construction of $(\widetilde{\theta}^*)^0_1$, the following diagram
		\[
		\xymatrix
		{
			\mathcal{V} \ar[r]^{\widetilde{\phi}} \ar[d]^{\widetilde{\sigma}^0_1}
			& \mathbb{L}^* \ar[d]^{(\widetilde{\theta}^*)^0_{1}} \\
			\sigma_Y^*\overline{\mathcal{V}} \ar[r]^{\sigma_Y^*\overline{\widetilde{\phi}}}
			& \sigma_Y^*\overline{\mathbb{L}^*}
		}
		\]
		commutes. By functoriality of symmetric powers and jet bundles, this implies 
		\begin{equation}\label{2.19}
			\widetilde{\theta}^{r-1}_{1-r} \circ \widetilde{\psi}_{r-1} = \sigma_Y^*\overline{\widetilde{\psi}_{r-1}} \circ \mathrm{sym}^{r-1}{\widetilde{\sigma}^0_1}.\tag{2.19}
		\end{equation}
		Hence we have the commutative diagram
		\[
		\xymatrix
		{
			\mathrm{sym}^{r-1}(\mathcal{V}) \ar[r]^{\widetilde{\psi}_{r-1}} \ar[d]^{\mathrm{sym}^{r-1}{\widetilde{\sigma}^0_1}}
			& J^{r-1}(\mathbb{L}^{1-r}) \ar[d]^{\widetilde{\theta}^{r-1}_{1-r}} \\
			\sigma_Y^*\overline{\mathrm{sym}^{r-1}(\mathcal{V})} \ar[r]^{\sigma_Y^*\overline{\widetilde{\psi}_{r-1}}}
			& \sigma_Y^*\overline{J^{r-1}(\mathbb{L}^{1-r})}
		}
		\]

		Since principal and sub-principal symbols commute with pull-back by $\sigma_Y$ and complex conjugation, the map $\widetilde{\mathcal{B}}$ preserves $\widetilde{\mathcal{U}}$.
		
		\medskip
		
		Let $D$ be a $\Gamma$-equivariant oper. By Theorem \ref{thm 1.18}, the unique flat vector bundle given by the sheaf of solution of $\widetilde{\delta}_{[D]} \in \mathrm{Diff}^r(\mathbb{L}^{1-r},\mathbb{L}^{1+r})$ with principal symbol $1$, is isomorphic to $J^{r-1}(\mathbb{L}^{1-r})$ equipped with $(\widetilde{\psi}_{r-1}^{-1})^*(D)$, see \cite[Section 4]{MR2036690}.
		
		Hence for any local section $v$ of $\mathbb{L}^{1-r}$,
		\begin{equation} \label{2.20}
			\widetilde{\delta}_{[D]}(j^rv)=0
			\quad \Longleftrightarrow \quad
			(\widetilde{\psi}_{r-1}^{-1})^*(D)(j^{r-1}v)=0. \tag{2.20}
		\end{equation}
		
		Therefore, to prove
		\[
		\widetilde{\psi}([\widetilde{\beta}(D)])
		=
		\widetilde{\mathcal B}(\widetilde{\delta}_{[D]}),
		\]
		it suffices to show that the two operators induce the same connection on $J^{r-1}(\mathbb{L}^{1-r})$.
		
		\medskip
		
		By definition,
		\[
		\widetilde{\beta}(D)
		= \mathrm{sym}^{r-1}(\widetilde{\sigma}^0_1)^*
		(\sigma_Y^*\overline{D}).
		\] 
		where $\mathrm{sym}^{r-1}(\widetilde{\sigma}^0_1)^* (\sigma_Y^*\overline{D})$ denotes the pull-back of $\sigma_Y^*\overline{D}$ by the isomorphism $\mathrm{sym}^{r-1}(\widetilde{\sigma}_1^0)$, defined in Remark \ref{remark 2.7.2}.
		
		Using (\ref{2.20}), $\widetilde{\delta}_{[\widetilde{\beta}(D)]}(j^rv)=0$ if and only if 
		\[ (\sigma_Y^*\overline{D}) \Big( \big( \mathrm{sym}^{r-1}(\widetilde{\sigma}^0_1) \circ \widetilde{\psi}_{r-1}^{-1} \big)(j^{r-1}v) \Big)=0. \]
		Using (\ref{2.19}), this is equivalent to
		\[
		(\sigma_Y^*\overline{D})
		\Big(
		\big((\sigma_Y^*\overline{\widetilde{\psi}_{r-1}})^{-1} \circ
		\widetilde{\theta}^{\,r-1}_{1-r}\big)
		(j^{r-1}v)
		\Big)=0.
		\]
		
		Since jet bundles commute with pullback by $\sigma_Y$ and conjugation, 
		\begin{align} \label{2.21}
			\widetilde{\theta}^r_{1-r}(j^rv)=j^r(\widetilde{\theta}^0_{1-r}(v)). \tag{2.21}
		\end{align}
		Hence the above condition is equivalent to
		\begin{align} \label{2.22}
			(D \circ \sigma_Y^*\overline{\widetilde{\psi}_{r-1}})
			\big(
			j^{r-1}(\sigma_Y^*\overline{\widetilde{\theta}^0_{1-r}(v)})
			\big)
			=0. \tag{2.22}
		\end{align}
		
		On the other hand,
		\begin{align*}
			\widetilde{\mathcal{B}} (\widetilde{\delta}_{[D]})= (\widetilde{\theta}^0_{1+r})^{-1} \circ \sigma_Y^* \overline{\widetilde{\delta}}_{[D]} \circ \widetilde{\theta}^r_{1-r}
		\end{align*}
		
		Thus $\widetilde{\mathcal{B}} (\widetilde{\delta}_{[D]})(v)=0$ if and only if
		\[ \widetilde{\delta}_{[D]}(\sigma_Y^*\overline{\widetilde{\theta}^r_{1-r}(j^rv)})=0.\]
		Using compatibility of jets, this becomes 
		\[\widetilde{\delta}_{[D]}(j^r\sigma_Y^*\overline{(\widetilde{\theta}^0_{1-r}v}))=0. \]
		By ($\ref{2.20}$), this is equivalent to 
		\begin{equation}\label{2.23}
			(D \circ \sigma_Y^*\overline{\widetilde{\psi}_{r-1}})
			\big(
			j^{r-1}(\sigma_Y^*\overline{\widetilde{\theta}^0_{1-r}(v)})
			\big)
			=0. \tag{2.23}
		\end{equation}
		The right-hand sides of (\ref{2.22}) and (\ref{2.23}) coincide.
		By uniqueness in Theorem~\ref{thm 1.18}, they are equal.
		This proves the commutativity of the diagram.
	\end{proof}

	\begin{corollary}
			The involution $\mathcal{B}$ defined in Corollary $\ref{corollary 2.12.1}$ coincides with the involution $\beta$ defined in Proposition $\ref{prop:2.6}$ on space of parabolic $\mathrm{SL}(r,\mathbb{C})$-opers on $X$.
	\end{corollary}
	\begin{proof}
		There is a natural bijection between the parabolic $\mathrm{SL}(r,\mathbb{C})$-opers on $X$ and the equivariant $\mathrm{SL}(r,\mathbb{C})$-opers on $Y$, proved in \cite{biswas2023parabolic}. Moreover, there is a bijective correspondence between $\mathcal{U} \subset H^0(X,\mathrm{Diff}^r\bigl( L_*^{1-r},L_*^{1+r}) \bigr)$ and $\mathcal{\widetilde{U}} \subset H^0(Y,\mathrm{Diff}^r\bigl( \mathbb{L}_*^{1-r},\mathbb{L}_*^{1+r}) \bigr)$. From the following diagram, it is easy to see that the outer rectangle commutes as the inner rectangle and other trapeziums commute:  
		   
	    \begin{center}
	    	\begin{tikzpicture}[scale=0.8, every node/.style={scale=0.85}]
	    		\node (TL_out) at (0, 5) {$\mathrm{SL}_X^p(r)$};
	    		\node (TR_out) at (6, 5) {$\mathrm{SL}_X^p(r)$};
	    		\node (BL_out) at (0, 0) {$\mathcal{U}$};
	    		\node (BR_out) at (6, 0) {$\mathcal{U}$};
	    		
	    		\node (TL_in)  at (1.8, 3.5) {$\mathrm{SL}_Y^\Gamma(r)$};
	    		\node (TR_in)  at (4.2, 3.5) {$\mathrm{SL}_Y^\Gamma(r)$};
	    		\node (BL_in)  at (1.8, 1.2) {$\widetilde{\mathcal{U}}$};
	    		\node (BR_in)  at (4.2, 1.2) {$\widetilde{\mathcal{U}}$};
	    		
	    		\draw[->, thick] (TL_out) -- (TR_out) node[midway, above] {$\beta$};
	    		\draw[->, thick] (TL_out) -- (BL_out);
	    		\draw[->, thick] (TR_out) -- (BR_out);
	    		\draw[->, thick] (BL_out) -- (BR_out) node[midway, below] {$\mathcal{B}$};
	    		
	    		\draw[->, thick] (TL_in) -- (TR_in) node[midway, below] {$\widetilde{\beta}$};
	    		\draw[->, thick] (TL_in) -- (BL_in) node[midway, right] {$\widetilde{\psi}$};
	    		\draw[->, thick] (TR_in) -- (BR_in) node[midway, right] {$\widetilde{\psi}$};
	    		\draw[->, thick] (BL_in) -- (BR_in) node[midway, below] {$\widetilde{\mathcal{B}}$};
	    		
	    		\draw[->, thick] (TL_out) -- (TL_in);
	    		\draw[->, thick] (TR_out) -- (TR_in);
	    		\draw[->, thick] (BL_out) -- (BL_in);
	    		\draw[->, thick] (BR_out) -- (BR_in);

	    	\end{tikzpicture}
	    \end{center}
	       This completes the proof.
	\end{proof}

\end{document}